\documentclass[12pt]{article}
\usepackage{latexsym,amssymb,amsmath,amsfonts,pictex,enumerate,amsthm,dsfont}
\usepackage{tikz}
\usepackage{scalerel,stackengine}
\usepackage[active]{srcltx}
\usepackage{breqn}

\newtheorem{theorem}{Theorem}[section] 
       
       \newtheorem{lemma}[theorem]{Lemma}

       \theoremstyle{definition} 

       \newtheorem{definition}[theorem]{Definition}

\stackMath
\newcommand\reallywidehat[1]{%
\savestack{\tmpbox}{\stretchto{%
  \scaleto{%
    \scalerel*[\widthof{\ensuremath{#1}}]{\kern-.6pt\bigwedge\kern-.6pt}%
    {\rule[-\textheight/2]{1ex}{\textheight}}
  }{\textheight}%
}{0.5ex}}%
\stackon[1pt]{#1}{\tmpbox}%
}
\begin{document}

\author{Stephen Deterding \\ Department of Mathematics, University of Kentucky
\thanks{Email address: deterding@uky.edu}}
\date{}

\title{Bounded point derivations on $R^p(X)$ and approximate derivatives}

\maketitle

\begin{abstract} It is shown that if a point $x_0$ admits a bounded point derivation on $R^p(X)$, the closure of rational function with poles off $X$ in the $L^p(dA)$ norm, for $p >2$, then there is an approximate derivative at $x_0$. A similar result is proven for higher order bounded point derivations. This extends a result of Wang which was proven for $R(X)$, the uniform closure of rational functions with poles off $X$.

\end{abstract}

\section{Introduction}

Let $X$ be a compact subset of the complex plane. Let $C(X)$ denote the set of all continuous functions on $X$ and let $R(X)$ be the subset of $C(X)$ that consists of all function in $C(X)$ which on $X$ are uniformly approximable by rational functions with poles off $X$. We denote by $R^p(X)$, $1 \leq p < \infty$, the closure of the rational functions with poles off $X$ in the $L^p$ norm where the underlying measure is $2$ dimensional Lebesgue (area) measure. It follows from H\"{o}lder's inequality that the uniform norm is more restrictive than the $L^p$ norm and thus $R(X) \subseteq R^p(X)$. 

\bigskip

The space $R^p(X)$ was originally studied as part of the following question of rational approximation: what are the necessary and sufficient conditions so that $R^p(X) = L^p(X)$? It is straightforward to show that $R^p(X) \neq L^p(X)$ unless $X$ has empty interior, so from now on, we will make this assumption. The following results are well known: if $1 \leq p<2$, then $R^p(X) = L^p(X)$ \cite{Sinanjan2}, and if $p \geq 2$ then there is a necessary and sufficient condition for $R^p(X) = L^p(X)$ involving Sobolev $q$-capacity \cite[Theorem 6]{Fernstrom} \cite{Hedberg2}


\bigskip

In this paper, we consider a different kind of approximation problem for $R^p(X)$. Since rational functions with poles off $X$ are smooth, but functions in $R^p(X)$ may not be differentiable at all, it is natural to ask how much is the differentiability of rational functions preserved under convergence in the $L^p$ norm. The primary tool for answering this question is that of a bounded point derivation. For a non-negative integer $t$, we say that $R^p(X)$ has a bounded point derivation of order $t$ at $x_0$ if there exists a constant $C> 0$ such that $|f^{(t)}(x_0)| \leq C ||f||_p$ for all rational functions $f$ with poles off $X$. If $t = 0$, we take the $0$-th order derivative to be the evaluation of the function at $x_0$. For this reason, a $0$-th order bounded point derivation is usually called a bounded point \textit{evaluation}. Bounded point evaluations have been widely studied in both rational approximation theory and operator theory. (See for instance \cite{Brennan3}, \cite{Fernstrom}, and \cite{Hedberg}) 

\bigskip

If $f$ is a function in $R^p(X)$ then there is a sequence $\{f_j\}$ of rational functions with poles off $X$ that converges to $f$ in the $L^p$ norm. If there is a bounded point derivation at $x_0$ then $|f_j^{(t)}(x_0) -f_k^{(t)}(x_0)| \leq C ||f_j -f_k||_p $, which tends to $0$ as $j$ and $k$ tend to infinity. Thus $\{f_j^{(t)}(x_0) \}$ is a Cauchy sequence and hence converges. Hence the map $f \to f^{(t)}(x_0)$ can be extended from the space of rational functions with poles off $X$ to a bounded linear functional on $R^p(X)$, which we denote as $D_{x_0}^t$. It follows that $D_{x_0}^t f = \displaystyle \lim_{j \to \infty} f_j^{(t)}(x_0)$, where $\{f_j\}$ is a sequence of rational functions which converges to $f$ in the $L^p$ norm. Note that the value of $D_{x_0}^t f$ does not depend on the choice of this sequence.

\bigskip



Thus bounded point derivations generalize the notion of a derivative to functions in $R^p(X)$ which may not be differentiable. In fact, it is a result of Dolzhenko \cite{Dolzhenko} that there is a nowhere differentiable function in $R(X)$, and hence also $R^p(X)$, whenever $X$ is a set with no interior. For this reason, it is important to understand the relationship between bounded point derivations and the usual notion of the derivative. This problem was first considered by Wang \cite{Wang} in the case of uniform rational approximation. Wang showed that the existence of a bounded point derivation on $R(X)$ at $x_0$ implies that every function in $R(X)$ has an approximate derivative at $x_0$. An approximate derivative is defined in the same way as the usual derivative, except that the limit is taken over a subset with full area density at $x_0$ rather than over all points of $X$. We recall what it means for a set to have full area density at $x_0$. Let $x_0$ be a point in the complex plane, let $\Delta_n(x_0)$ denote the ball centered at $x_0$ with radius $\frac{1}{n}$ and let $m$ denote $2$ dimensional Lebesgue measure. A set $E$ is said to have  full area density at $x_0$ if$ \displaystyle  \lim_{n \to \infty} \dfrac{m (\Delta_n(x_0) \setminus E)}{m(\Delta_n(x_0))} = 0$. Wang also proved a similar result for higher order bounded point derivations. The goal of this paper is to extend Wang's results to functions in $R^p(X)$. Our first result is the following theorem.

\begin{theorem}
\label{main}

For $2 < p < \infty$, suppose that there is a bounded point derivation on $R^p(X)$ at $x_0$ denoted by $D_{x_0}^1$. Then given a function $f$ in $R^p(X)$, there exists a set $E$ of full area density at $x_0$ such that 
\[ \lim_{x \to x_0, x \in E} \left|\dfrac{f(x) - f(x_0)}{x-x_0} - D_{x_0}^1 f\right| = 0 \]

\end{theorem}

\bigskip

We remark that this theorem is only valid for $2 < p < \infty$. Recall that when $1 \leq p < 2$, $R^p(X) = L^p(X)$ and thus there are no bounded point derivations on $R^p(X)$. In fact there are not even bounded point \textit{evaluations} \cite[Lemma 3.5]{Brennan3}. This still leaves open the case of $p =2$. It is possible for bounded point derivations on $R^2(X)$ to exist; however, we do not know whether Theorem \ref{main} still holds for $R^2(X)$.

\bigskip

We will also prove the following higher order extension of Theorem \ref{main}. The quantity $\Delta_h^t f(x_0)$ is the $t$-th order difference quotient of $f$ at $x_0$ and $h$, which is defined in the next section.

\begin{theorem}
\label{higherorder}

Let $t$ be a positive integer. For $2 < p < \infty$ suppose that there exists a bounded point derivation of order $t$ on $R^p(X)$ at $x_0$ denoted by $D_{x_0}^t$. Then given a function $f$ in $R^p(X)$ there exists a set $E'$ with full area density at $0$, such that 
\[ \lim_{h \to 0, h \in E'} \left|\Delta_h^t f(x_0) -D_{x_0}^tf \right|= 0\]

\end{theorem}

\bigskip

The outline of the rest of the paper is as follows. In the next section we consider higher order difference quotients and approximate derivatives. In Section $3$ we briefly review a few concepts from measure theory which are fundamental to our proofs, and Section $4$ is devoted to the construction of a set of full area density at $x_0$ which is needed for the proof of the main result. We present the proofs of Theorems \ref{main} and \ref{higherorder} in Sections $5$ and $6$ respectively.

\section{Higher order approximate derivatives}

 Intuitively, a higher order approximate derivative at $x_0$ should be defined in the same way as a higher order derivative except that the limit of the difference quotient should be taken over a set with full area density at $x_0$. However, a function in $R^p(X)$ may not have derivatives of any orders and thus we cannot define an approximate higher order derivative in terms of any of the lower order derivatives. Hence we will use the following definition for higher order difference quotients.

\begin{definition}
\label{t-th diff quot}
Let $t$ be a positive integer, let $f$ be a function in $R^p(X)$, let $x_0$ be a point in $X$, and choose $h \in \mathds{C}$ so that $f$ is defined at $x_0 + sh$ for $s = 0, 1, ..., t$.  The { \bf $t$-th order difference quotient of $f$ at $x_0$ and $h$} is denoted by $\Delta_h^t f(x_0)$ and defined by 

\[\Delta_h^t f(x_0) =  h^{-t} \sum_{s = 0}^{t} (-1)^{t-s} \binom{t}{s} f(x_0 + sh) \]

\end{definition}

\bigskip

For this definition to be reasonable, it should agree with the usual definition for higher order derivatives when $f$ has derivatives of all orders. 

\begin{theorem}
\label{diff quot thm}

Suppose that $f$ has derivatives of all orders on a neighborhood of $x_0$. Then for all positive integers $t$, $f^{(t)}(x_0) = \displaystyle \lim_{h \to 0} \Delta_h^t f(x_0)$.

\end{theorem}

\begin{proof}

The proof is by induction. Since $\Delta_h^1 f(x_0) = \dfrac{f(x_0+h)-f(x_0)}{h}$ the theorem is true for $t = 1$. Now assume that $f^{(t-1)}(x_0) = \displaystyle \lim_{h \to 0} \Delta_h^{t-1} f(x_0)$. Then 

\[f^{(t)}(x_0) = \displaystyle \lim_{h \to 0} \dfrac{  \Delta_h^{t-1} f(x_0+h) -  \Delta_h^{t-1} f(x_0)}{h}  = \lim_{h \to 0} \Delta_h^1 \circ \Delta_h^{t-1}f(x_0)\] 

\bigskip

Thus to show that $f^{(t)}(x_0) = \displaystyle \lim_{h \to 0} \Delta_h^{t} f(x_0)$ it is enough to prove that $\Delta_h^1 \circ \Delta_h^{t-1}f(x_0) = \Delta_h^{t}f(x_0)$. It follows from Definition \ref{t-th diff quot} that

\[ \Delta_h^1 \circ \Delta_h^{t-1}f(x_0) =  h^{-t} \left\{ \sum_{s = 0}^{t-1} (-1)^{t-1-s} \binom{t-1}{s} f(x_0 + (s+1)h) -  \sum_{s = 0}^{t-1} (-1)^{t-1-s} \binom{t-1}{s} f(x_0 + sh)\right\}\]

\bigskip

A change of variable of $s = s-1$ in the first sum yields

\[ \Delta_h^1 \circ \Delta_h^{t-1}f(x_0) =  h^{-t} \left\{\sum_{s = 1}^{t} (-1)^{t-s} \binom{t-1}{s-1} f(x_0 + sh) -  \sum_{s = 0}^{t-1} (-1)^{t-1-s} \binom{t-1}{s} f(x_0 + sh)\right\}\]

\bigskip
Multiplying the second sum by $(-1)$ changes the subtraction to addition.
Then moving the $t$-th term of the first sum outside the sum and doing the same to the $0$-th term of the second sum yields

\[ 
\begin{split}
\Delta_h^1 \circ \Delta_h^{t-1}f(x_0) &=  h^{-t} \left\{ f(x_0+ th) + \sum_{s = 1}^{t-1} (-1)^{t-s} \binom{t-1}{s-1} f(x_0 + sh)  \right.\\  &\quad \left. {} +\sum_{s = 1}^{t-1} (-1)^{t-s} \binom{t-1}{s} f(x_0 + sh) + (-1)^{t} f(x_0)\vphantom{ h^{-t}}\right\}
\end{split}
\]

\bigskip

The two sums can be combined using the binomial identity $\binom{t-1}{s-1} + \binom{t-1}{s} = \binom {t}{s}$. Hence

\[ \Delta_h^1 \circ \Delta_h^{t-1}f(x_0) =  h^{-t} \left\{ f(x_0+ th) + \sum_{s = 1}^{t-1} (-1)^{t-s} \binom{t}{s} f(x_0 + sh) + (-1)^{t} f(x_0)\right\}\]

\bigskip

 In addition since $\binom{t}{0} = 1$ and $\binom{t}{t} =1$ the two terms outside the sum can be put back into the sum and thus

\[ \Delta_h^1 \circ \Delta_h^{t-1}f(x_0) = h^{-t} \sum_{s = 0}^{t} (-1)^{t-s} \binom{t}{s} f(x_0 + sh) = \Delta_h^tf(x_0)\]

\end{proof}

\bigskip

We now define higher order approximate derivatives using Definition \ref{t-th diff quot}.

\begin{definition}

Let $t$ be a positive integer. A function $f$ in $R^p(X)$ has an {\bf approximate derivative of order $t$} at $x_0$ if there exists a set $E'$ with full area density at $0$, and a number $L$ such that 

\[ \lim_{h \to 0, h \in E'} \left|\Delta_h^t f(x_0) - L\right| = 0\]

\bigskip

\noindent
We say that $L$ is the approximate derivative of order $t$ at $x_0$.

\end{definition}

Thus a $t$-th order approximate derivative at $x_0$, is a $t$-th order difference quotient in which the limit as $h$ tends to $0$ is taken over a set with full area density at $0$. The reason that the set $E'$ has full area density at $0$ instead of at $x_0$ is that the limits in the definitions of usual higher order derivatives are taken as $h$ tends to $0$ and therefore, the higher order approximate derivatives must be defined similarly.

\section{Results from measure theory}
 In this section, we briefly review some results from measure theory to be used in our proofs. From now on $q$ denotes the conjugate exponent to $p$; that is, $ q=  \frac{p}{p-1}$, and $dA$ denotes $2$ dimensional Lebesgue (area) measure. Since a bounded point derivation is a bounded linear functional, there exists a function $k$ in $L^q(X)$ such that the measure $kdA$ represents the bounded point derivation. If the representing measure for a $t$-th order bounded point derivation on $R^p(X)$ is known, then it would be useful to have a method for finding the representing measures for bounded point derivations of lesser orders. The next lemma, which describes such a method, is based on a theorem of Wilken \cite{Wilken}.

\begin{lemma}
\label{Wilken2}

Let $1 \leq p < \infty$. Let $t$ be a positive integer and suppose that there is a $t$-th order bounded point derivation on $R^p(X)$ at $x_0$ with representing measure $k_t dA$. For each $m$ with $0 \leq m \leq t$, let $k_m = \frac{m!}{t!}(z-x_0)^{t-m} k_t$. Then $k_m$ belongs to $L^q(X)$ and $k_m dA$ represents an $m$-th order bounded point derivation on $R^p(X)$ at $x_0$.

\end{lemma}

\begin{proof}

Since $k_t$ belongs to $L^q(X)$, $k_m$  also belongs to $L^q(X)$. To prove that $k_m$ represents an $m$-th order bounded point derivation on $R^p(X)$ at $x_0$, we first suppose that $f$ is a rational function with poles off $X$. Hence $f(z)(z-x_0)^{t-m}$ is a rational function and integrating $f(z)(z-x_0)^{t-m}$ against the measure $k_t dA$ is the same as evaluating the $t$-th derivative of $f(z)(z-x_0)^{t-m}$ at $z = x_0$, which can be done using the general Leibniz rule. The only term that will not vanish is the term which puts exactly $t-m$ derivatives on $(z-x_0)^{t-m}$ and $m$ derivatives on $f(z)$. Hence 

\[ \int f(z) (z-x_0)^{t-m} k_t(z) dA_z = \binom{t}{m} (t-m)! f^{(m)}(x_0) =\dfrac{t!}{m!} f^{(m)}(x_0) \]

and \[ \int f(z) k_m(z) dA_z = f^{(m)}(x_0)\]

\bigskip

\noindent Hence by H\"{o}lder's inequality, $|f^{(m)}(x_0)| \leq ||k_m||_q ||f||_p$. So there is a bounded point derivation of order $m$ at $x_0$ and the measure $k_m dA$ represents the bounded point derivation.

\end{proof}






\bigskip

Lastly, we review the definitions of the Cauchy transform and Newtonian potential of a measure.

\begin{definition}
\label{CauchyTransform}

Let $k \in L^q(X)$.

\begin{enumerate}

\item The \textbf{Cauchy transform} of  the measure $k dA$, which is denoted by $\hat{k}(x)$ is defined by 

\[ \hat{k}(x) = \int \dfrac{k(z)}{z-x}dA_z\]

\item The \textbf{Newtonian potential} of the measure $k dA$, which is denoted by $\tilde{k}(x)$ is defined by

\[ \tilde{k}(x) = \int \dfrac{|k(z)|}{|z-x|}dA_z\]

\end{enumerate}

\end{definition}

 \section{A set with full area density at $x_0$}
 
 In this section a method is given to construct a set with full area density at $x_0$ which also possesses the properties needed for the proofs of Theorems \ref{main} and \ref{higherorder}. Constructing this set can be accomplished by first listing the desired properties and then showing that the set with these desired properties has full area density at $x_0$.

\begin{theorem}
\label{exceptional set}

 Suppose $1 < q < 2$. Let $k \in L^q(X)$, and let $0< \delta_0 < 1 $. Let $E$ be the set of $x$ in $X$ that satisfy the following properties.

\begin{enumerate}

\item $\displaystyle \int_X \dfrac{|(x-x_0)k(z)|^q}{|z-x|^q } dA < \delta_0$

\item $|x-x_0|\tilde{k}(x) < \delta_0$

\end{enumerate}

\bigskip

Then $E$ has full area density at $x_0$.

\end{theorem}

\bigskip

To prove Theorem \ref{exceptional set}, we will need a few lemmas. The first lemma is an extension of a result of Browder \cite[Lemma 1]{Browder}.

\begin{lemma}
\label{Browder2}

Suppose $1 < q < 2$. Let $\chi_{\{x_0\}}$ be the characteristic function of the point $x_0$ and let $m$ denote $2$ dimensional Lebesgue measure. For $n$ positive, let $\Delta_n = \{x: |x-x_0|< \frac{1}{n}\}$ and let $w_n(z) = \displaystyle \frac{1}{m(\Delta_n)} \int_{\Delta_n} \dfrac{|x-x_0|^q}{|z-x|^q} dm_x$. Then $w_n(z) \leq \dfrac{2}{2-q}$ for all $z$ and all $n$, and $w_n(z)$ converges to $\chi_{\{x_0\}}$ pointwise as $n \to \infty$.

\end{lemma}

\begin{proof}

We first show that $w_n(z)$ converges to $\chi_{\{x_0\}}$ pointwise as $n \to \infty$.  If $z = x_0$, then the integrand is identically $1$ and $w_n(z) = 1$ for all $n$. Now suppose that $z \neq x_0$. If $n$ is sufficiently large, then $|z-x_0| > \frac{1}{n}$ and thus $z$ need not be in $\Delta_n$ for large $n$. Since the measure of $\Delta_n$ is $\frac{\pi}{n^2}$, we can rewrite $w_n(z)$ as $ \displaystyle  \frac{n^2}{\pi} \int_{\Delta_n} \dfrac{|x-x_0|^q}{|z-x|^q} dm_x$. In addition since $x$ belongs to $\Delta_n$, $|x-x_0| \leq \frac{1}{n}$. Therefore $ \displaystyle w_n(z) \leq \frac{n^{2-q}}{\pi} \int_{\Delta_n} \dfrac{1}{|z-x|^q} dm_x$. If $n$ is sufficiently large, it follows from the reverse triangle inequality that

\[|z-x| \geq \Bigr||z-x_0| - |x_0 -x|\Bigr| \geq |z-x_0|-\frac{1}{n} > 0\]

\bigskip

\noindent Thus $|z-x|^q > (|z-x_0|-\frac{1}{n})^q>0$ and 

\[w_n(z) \leq \dfrac{n^{2-q}}{\pi(|z-x_0|-\frac{1}{n})^q}\int_{\Delta_n}  dm_x \leq \dfrac{ n^{-q}}{(|z-x_0|-\frac{1}{n})^q}\]

\bigskip

\noindent which tends to $0$ as $n\to \infty$. Thus if $z \neq x_0$ then $w_n(z)$ tends to $0$ pointwise as $n\to \infty$ and hence $w_n(z)$ converges to $\chi_{\{x_0\}}$ pointwise as $n\to \infty$.

\bigskip

To show that $w_n(z) \leq \dfrac{2}{2-q}$ for all $z$ and all $n$, we first recall the inequality \[ \displaystyle w_n(z) \leq \frac{n^{2-q}}{\pi} \int_{\Delta_n} \dfrac{1}{|z-x|^q} dm_x\]

\bigskip

\noindent which was proved above. Now, the value of the integral would be larger if the integration was performed over $B(z,\frac{1}{n})$, the disk with radius $\frac{1}{n}$ centered at $z$ instead of integrating over $\Delta_n$. Hence,

\[ w_n(z) \leq \frac{n^{2-q}}{\pi} \int_{B(z,\frac{1}{n})} \dfrac{1}{|z-x|^q} dm_x \]

\bigskip

It follows from a calculation that $\displaystyle \int_{B(z,\frac{1}{n})} \dfrac{1}{|z-x|^q} dm_x = \dfrac{2 \pi n^{-(2-q)}}{2-q}$. Hence $w_n(z) \leq  \dfrac{2}{2-q}$.

\end{proof}

\bigskip

We note that it is in the above lemma, that our proof breaks down for the case of $p =2$. If $p =2$, then $q=2$, but $w_n(z)$ is no longer bounded in this case since $\frac{1}{z^2}$ is not locally integrable.

\bigskip

\begin{lemma}
\label{Browder}

Suppose $1 < q < 2$. Let $\Delta_n = \left\{x \in X :|x-x_{0}| <\frac{1}{n}\right\}$, let $k \in L^q(X)$ and let $m$ denote $2$ dimensional Lebesgue measure. Then

\[ \frac{1}{m(\Delta_n)} \int_{\Delta_n} \left\{\int_X \dfrac{|x-x_0|^q|k(z)|^q}{|z-x|^q} dm_z \right\} dm_x \to 0 \]

\noindent as $n\to \infty$.

\end{lemma}

\begin{proof}
 Let $w_n(z)$ be as in the previous lemma. Since $w_n(z)$ is uniformly bounded for all $n$, $ \displaystyle \int_X w_n(z) |k(z)|^q \leq C \int_X  |k(z)|^q $ and because $k(z) \in L^q(X)$, it follows that this integral is bounded. Since $w_n(z)$ tends to $0$ almost everywhere as $n\to \infty$, it follows from the dominated convergence theorem that $\displaystyle \int_{X} w_n(z) |k(z)|^q \to 0$ as $n\to \infty$. Recall that $w_n(z) = \displaystyle \frac{1}{m(\Delta_n)} \int_{\Delta_n} \dfrac{|x-x_0|^q}{|z-x|^q} dm_x$. Hence interchanging the order of integration yields

 \[\frac{1}{m(\Delta_n)} \int_{\Delta_n} \left\{\int_X \dfrac{|x-x_0|^q|k(z)|^q}{|z-x|^q} dm_z \right\} dm_x \to 0\]

\bigskip

\noindent as $n\to \infty$.
\end{proof}

\begin{lemma}
\label{property2}

Suppose $1 < q < 2$. Choose $\delta >0$, let $k \in L^q(X)$ and let $m$ denote 2 dimensional Lebesgue measure. Let \[E_{\delta} = \left\{ x \in X : \int_X \dfrac{|x-x_0|^q|k(z)|^q}{|z-x|^q} dm_z < \delta \right\}\] Then $E_{\delta}$ has full area density at $x_0$.

\end{lemma}

\begin{proof}

It follows immediately from the definition of $E_{\delta}$ that

\[\frac{1}{m(\Delta_n)}\int_{\Delta_n \setminus E_{\delta}} \left\{ \int_X \dfrac{|x-x_0|^q|k(z)|^q}{|z-x|^q} dm_z \right\} dm_x\geq \dfrac{\delta m(\Delta_n \setminus E_{\delta})}{m(\Delta_n)}\]

\bigskip

By Lemma \ref{Browder} the left hand side tends to $0$ as $n$ goes to infinity. Thus $\displaystyle \lim_{n \to \infty} \dfrac{ m(\Delta_n \setminus E_{\delta})}{m(\Delta_n)} = 0$ and $E_{\delta}$ has full area density at $x_0$.

\end{proof}

\bigskip

The proof of Theorem \ref{exceptional set} now follows from Lemma \ref{property2}.

\begin{proof}\textbf{(Theorem \ref{exceptional set})}

Lemma \ref{property2} immediately implies that the set of $x$ in $X$ where property $1$ holds has full area density at $x_0$. To show that the set where property $2$ holds also has full area density at $x_0$ note that by H\"{o}lder's inequality

\[ \int_X \dfrac{|x-x_0||k(z)|}{|z-x|} dm_z \leq \left\{\int_X \dfrac{|x-x_0|^q|k(z)|^q}{|z-x|^q} dm_z \right\}^{\frac{1}{q}} \cdot m(X)^{\frac{1}{p}}  \]

\bigskip

It follows from Lemma \ref{property2} that the integral on the right is bounded. If $m(X) = 0$, then property $2$ holds for any choice of $\delta_0 > 0$ and we are done. Thus we can assume that $m(X) \neq 0$. If the integral on the right hand side is less that $\dfrac{\delta_0}{m(X)^{\frac{1}{p}}}$ then the left hand side will be less than $\delta_0$. This can be done by choosing $\delta = \dfrac{\delta_0}{m(X)^{\frac{1}{p}}} $ in Lemma \ref{property2}. Thus property $2$ also holds on a set with full area density at $x_0$ and thus the set $E$ has full area density at $x_0$. 
\end{proof}

\bigskip

\section{The proof of Theorem \ref{main}}

The goal of this section is to prove Theorem \ref{main} by showing that, for $2 < p < \infty$, the existence of a bounded point derivation on $R^p(X)$ at $x_0$ implies that every function in $R^p(X)$ has an approximate derivative at $x_0$. Choose $f$ in $R^p(X)$ and let $g(z) = f(z) - D_{x_0}^0 f - D_{x_0}^1 f\cdot (x-x_0)$. Then to show that $f(z)$ has an approximate derivative at $x_0$, it suffices to show that $g(z)$ has an approximate derivative at $x_0$ since $g(z)$ differs from $f(z)$ by a polynomial. The reason that it is more advantageous to work with $g(z)$ rather than $f(z)$ is that $D_{x_0}^0 (g) = D_{x_0}^1(g) = 0$.

\bigskip

Consider the following family of linear functionals defined for every $x \in X$: $L_x(F) = \dfrac{F(x)}{x-x_0}-D_{x_0}^1F$. To prove Theorem \ref{main}, it suffices to show that there is a set $E$ with full area density at $x_0$ such that $L_x(g)$ tends to $0$ as $x$ tends to $0$ through the points of $E$. Once this is shown, it follows that $\displaystyle \lim_{x \to x_0}  \dfrac{g(x)}{x-x_0}-D_{x_0}^1g = 0$ and since $g(x_0) = 0$, this shows that $g$ has an approximate derivative at $x_0$.

\bigskip

 Since $R^p(X)$ has a bounded point derivation at $x_0$, there exists a function $k_1$ in $L^q(X)$ such that the measure $k_1dA$ represents the bounded point derivation. Hence by Lemma \ref{Wilken2}, the function $k = (z-x_0)k_1$ belongs to $L^q(X)$ and $kdA$ is a representing measure for $x_0$. Fix $0 < \delta_0 < 1$ and let $E$ be the set of $x$ in $X$ that satisfies the following properties.

\begin{enumerate}

\item $\displaystyle \int_X \dfrac{|(x-x_0)k_1|^q}{|z-x|^q } dA < \delta_0$

\item $\displaystyle \int_X \dfrac{|(x-x_0)k|^q}{|z-x|^q } dA < \delta_0$

\item $|x-x_0|\tilde{k}(x) < \delta_0$

\end{enumerate}

\bigskip

\noindent It follows from Theorem \ref{exceptional set} that $E$ has full area density at $x_0$.

\bigskip

To show that $L_x(g)$ tends to $0$ through $E$ it is useful to consider how $g(z)$ can be approximated by rational functions with poles off $X$. Since $f$ is in $R^p(X)$, there is a sequence $\{f_j\}$ of rational functions with poles off $X$ which converges to $f(z)$ in the $L^p$ norm. Let $g_j(z) = f_j(z) - D_{x_0}^0 f_j -  D_{x_0}^1 f_j \cdot (x-x_0)$. Then $\{g_j\}$ is a sequence of rational functions with poles off $X$ that possesses the following properties:

\begin{enumerate}

\item $\{g_j\}$ converges to $g(z)$ in the $L^p$ norm.

\item For each $j$, $D_{x_0}^0 g_j= D_{x_0}^1g_j = 0$.

\item $L_x(g_j)$ converges to $0$ as $x$ tends to $x_0$.

\end{enumerate}

\noindent The first two properties are easy to verify. The third property follows since $g_j(z)$ is a rational function with poles off $X$ and thus $D_{x_0}^1g_j = g_j'(x_0)$.

\bigskip

It now follows from the linearity of $L_x$ and the triangle inequality that $|L_x(g)| \leq |L_x(g-g_j)| + |L_x(g_j)|$. Hence to show that $L_x(g)$ tends to $0$ as $x$ tends to $x_0$, it follows from property $3$ that it is enough to show that $L_x(g-g_j) \to 0$ as $j\to \infty$. By property $1$ it suffices to prove that there is a constant $C$ which does not depend on $x$ such that for all $x$ in $E$, $|L_x(g-g_j)| \leq C ||g-g_j||_p$. Moreover, since a bounded point derivation is already a bounded linear functional, it is enough to show that there is a constant $C$ which does not depend on $j$ such that $\left|\dfrac{g(x)-g_j(x)}{x-x_0}\right| \leq C ||g-g_j||_p $. This is done in Lemma \ref{Browder1}

\bigskip

We will first need to construct a representing measure for $x$ in $E$, which allows  $\dfrac{g(x)-g_j(x)}{x-x_0}$ to be expressed by an integral, from which the desired bound can be obtained. To do this, we borrow a technique of Bishop \cite{Bishop}. Bishop showed that if $\mu$ is an annihilating measure on $R(X)$ (i.e $\int f d\mu = 0$ for all $f$ in $R(X)$) and if the Cauchy transform $\hat{\mu}(x)$ is defined and nonzero, then the measure defined by $\dfrac{1}{\hat{\mu}(x)} \dfrac{\mu(z)}{z-x}$ is a representing measure for $x$. If $kdA$ is a representing measure for $x_0$ on $R^p(X)$ then $(z-x_0)kdA$ is an annihilating measure on $R^p(X)$ and thus Bishop's technique can be used to construct a representing measure for $x$ on $R^p(X)$.

\begin{lemma}
\label{Bishop1}

Let $k$ be a function in $L^q(X)$ such that $kdA$ is a representing measure for $x_0$. Choose $x$ in $X$ and suppose that $|x-x_0| \tilde{k} < \delta < 1$, and that $\dfrac{(x-x_0)k}{z-x}$ belongs to $L^q(X)$. Let $c = \reallywidehat{(z-x_0)k}(x)$ and let $k_x(z)  = \dfrac{1}{c}\dfrac{(z-x_0)k(z) }{z-x}$. Then there exists a bounded point evaluation on $R^p(X)$ at $x$ and $k_x dA$ is a representing measure for $x$.

\end{lemma}

\bigskip

\begin{proof} 

Before we begin the proof, we note a few things. First

\[\displaystyle c = \reallywidehat{(z-x_0)k}(x) = \int \dfrac{(z-x_0)k}{z-x} dA_z = 1 + \int \dfrac{(x-x_0)k}{z-x} dA_z = 1+(x-x_0)\hat{k}(x)\]

\bigskip
\noindent Thus $1-|x-x_0|\tilde{k}(x) \leq |c| \leq 1+ |x-x_0|\tilde{k}(x)$ and hence, $1-\delta \leq |c|\leq 1+\delta$. Since $\delta < 1$, $k_x$ is well defined. Second, $k_x$ can also be written as follows:

\begin{equation*}
\label{k_x}
k_x(z) = \dfrac{(z-x_0)k(z)}{(z-x)(1+(x-x_0) \hat{k}(x))}. 
\end{equation*}

\bigskip

\noindent Finally, $\dfrac{(z-x_0)k(z)}{z-x} = 1 + \dfrac{(x-x_0)k(z)}{z-x}$ and hence $k_x$ belongs to $L^q(X)$.

\bigskip

If $F$ is a rational function with poles off $X$, $\dfrac{[F(z)-F(x)](z-x_0)}{z-x}$ is also a rational function with poles off $X$. Since $k dA$ is a representing measure for $x_0$,  $\displaystyle \int \dfrac{[F(z)-F(x)](z-x_0)}{z-x} k(z) dA_z = 0$ and hence \[
\displaystyle \int \dfrac{F(z)(z-x_0)}{z-x} k(z) dA_z - \int \dfrac{F(x)(z-x_0)}{z-x} k(z) dA_z = 0.\]

\bigskip
 \noindent Since $z-x_0 = z-x + x-x_0$, it follows that

\[ \int \dfrac{F(x)(z-x_0)}{z-x} k(z) dA_z = \int F(x)k(z) dA_z + \int \dfrac{F(x)(x-x_0)k(z)}{z-x}dA_z = F(x)(1+(x-x_0)\hat{k}(x)).\]

\bigskip

\noindent Hence $ \displaystyle F(x) = \int \dfrac{F(z)(z-x_0)k(z)}{(z-x)(1+(x-x_0) \hat{k}(x))}dA_z$. So  $F(x) = \int F(z) k_x(z) dA$ whenever $F$ is a rational function with poles off $X$. Thus by H\"{o}lder's inequality $|F(x)| \leq ||k_x||_q ||F||_p$ and since $k_x$ is an $L^q$ function, it follows that $x$ admits a bounded point evaluation on $R^p(X)$ and that $k_x dA$ is a representing measure for $x$.

\end{proof}

\bigskip

\begin{lemma}
\label{Browder1}

Suppose that $x$ belongs to $E$ and let $j$ be a positive integer. Then there exists a constant $C$ which does not depend on $x$ or $j$ such that $\dfrac{|g(x)-g_j(x)|}{|x-x_0|}  \leq C  ||g-g_j||_p$.

\end{lemma}

\begin{proof}

\bigskip

 If $x$ belongs to $E$, then the hypotheses of Lemma \ref{Bishop1} are satisfied and $k_xdA$ is a representing measure for $x$. Thus




\[ |g(x)-g_j(x)| = \dfrac{1}{|c|}\left| \int [g(z)-g_j(z)] \left(\dfrac{z-x_0}{z-x} \right)k(z) dA_z\right| \]

\bigskip

\noindent Since $D^0_{x_0} [g(z)-g_j(z)] = 0$, it follows that $\int [g(z)-g_j(z)] k(z) dA_z = 0$. Then since $\dfrac{z-x_0}{z-x} = 1 + \dfrac{x-x_0}{z-x} $ we obtain that

\begin{dmath}
\label{eq1.5}
|g(x)-g_j(x)| =  \dfrac{|x-x_0|}{|c|}\left| \int [g(z)-g_j(z)] \dfrac{k(z)}{z-x} dA_z\right|
\end{dmath}

\bigskip

\noindent Next, observe that $\dfrac{1}{z-x}= \dfrac{1}{z-x_0} + \dfrac{x-x_0}{(z-x)(z-x_0)}$. Applying this observation to  \eqref{eq1.5} yields

\begin{dmath}
\label{eq2}
|g(x)-g_j(x)|  =  \dfrac{|x-x_0|}{|c|} \left| \int [g(z)-g_j(z)]\dfrac{k(z)}{z-x_0 }  dA_z+ \int [g(z)-g_j(z)] \dfrac{(x-x_0) k(z)}{(z-x)(z-x_0)}  dA_z \right|
\end{dmath}

\bigskip

 \noindent The first integral in \eqref{eq2} is the same as the bounded point derivation at $x_0$ applied to $g(z)-g_j(z)$ which is $0$,  and hence $ \displaystyle |g(x)-g_j(x)| = \dfrac{|x-x_0|}{|c|} \left|\int [g(z)-g_j(z)] \dfrac{(x-x_0)k_1(z)}{(z-x)}  dA_z \right|$. Finally by H\"{o}lder's inequality, 

\[\dfrac{|x-x_0|}{|c|} \left|\int [g(z)-g_j(z)] \dfrac{(x-x_0)k_1(z)}{(z-x)}  dA_z \right| \leq \dfrac{|x-x_0|}{|c|} \left \|g-g_j \right \|_p \left \|\dfrac{(x-x_0)k_1}{(z-x)}\right \|_q\]

\bigskip

\noindent and since it follows from property $1$ of $E$ that $\left \|\dfrac{(x-x_0)k_1}{(z-x)}\right \|_q \leq \delta_0$, there is a constant $C$ that does not depend on $x$ or $j$ such that 

\[
|g(x) - g_j(x)| \leq C |x-x_0| \cdot||g-g_j||_p. 
\]

\end{proof}

\section{Higher order bounded point derivations}

The goal of this section is to prove Theorem \ref{higherorder} by modifying the proof of Theorem \ref{main}. Choose $f$ in $R^p(X)$ and let $g(z) = f(z) - D_{x_0}^0 f - D_{x_0}^1 f \cdot (z-x_0) - ... - \dfrac{1}{t!} D_{x_0}^t f \cdot (z-x_0)^t$. As before, to show that $f(z)$ has a $t$-th order approximate derivative at $x_0$ it suffices to show that $g(z)$ has a $t$-th order approximate derivative at $x_0$. Also note that  $D_{x_0}^mg = 0$ for $0 \leq m \leq t$. 

\bigskip

 Consider the following family of  linear functionals defined for every $h$ in $\mathbb{C}$: $L_h(F) = \Delta_h^tF(x_0) - D_{x_0}^tF$. To prove Theorem \ref{higherorder}, it suffices to show that there is a set $E'$ with full area density at $0$ such that $L_h(g)$ tends to $0$ as $h$ tends to $0$ through the points of $E'$. Once this is shown, it follows that  $\displaystyle \lim_{h \to 0, h \in E'} \left|\Delta_h^t g(x_0) - D_{x_0}^t g\right| = 0$ and thus $g$ has a $t$-th order approximate derivative at $x_0$. 


\bigskip

Since there is a $t$-th order bounded point derivation on $R^p(X)$ at $x_0$, there exists a function $k_t$ in $L^q(X)$ such that the measure $k_tdA$ represents this $t$-th order bounded point derivation. Hence by Lemma \ref{Wilken2}, the function $k = \dfrac{(z-x_0)k_t}{t!}$ belongs to $L^q(X)$ and $kdA$ is a representing measure for $x_0$. Fix $0 < \delta_0 < 1$ and let $E$ be the set of $x$ in $X$ that satisfies the following properties.

\begin{enumerate}

\item $\displaystyle \int_X \dfrac{|(x-x_0)k_t|^q}{|z-x|^q } dA < \delta_0$

\item $\displaystyle \int_X \dfrac{|(x-x_0)k|^q}{|z-x|^q } dA < \delta_0$

\item $|x-x_0|\tilde{k}(x) < \delta_0$

\end{enumerate}

\bigskip

It follows from Theorem \ref{exceptional set} that $E$ has full area density at $x_0$. Now, for $1 \leq s \leq t$, let $E_s = \{h \in \mathbb{C}: x_0+sh \in E \}$ and let $E' = \displaystyle \bigcap_{s=1}^t E_s $. Then for each $s$, $E_s$ has full area density at $0$ and hence $E'$ also has full area density at $0$.

\bigskip

As in the previous section, to show that $L_h(g)$ tends to $0$ through $E'$ it is useful to consider how $g(z)$ can be approximated by rational functions with poles off $X$. Since $f$ belongs to $R^p(X)$, there is a sequence $\{f_j\}$ of rational functions with poles off $X$ which converges to $f(z)$ in the $L^p$ norm. Let $g_j(z) = f_j(z) - D_{x_0}^0 f_j -  D_{x_0}^1 f_j \cdot (x-x_0) - ... - \dfrac{1}{t!} D_{x_0}^t f_j \cdot (x-x_0)^t$. Then $\{g_j\}$ is a sequence of rational functions with poles off $X$ that possesses the following properties.

\begin{enumerate}

\item $\{g_j\}$ converges to $g(z)$ in the $L^p$ norm.

\item For each $j$, $D_{x_0}^m g_j = 0$ for $0 \leq m \leq t$.

\item $L_h(g_j)$ converges to $0$ as $h$ tends to $0$.

\end{enumerate}

\bigskip

\noindent The first two properties are easy to verify. The third property follows since $g_j(z)$ is a rational function with poles off $X$ and thus $D_{x_0}^tg_j = g_j^{(t)}(x_0)$. 

\bigskip

It now follows from the linearity of $L_h$ and the triangle inequality that $|L_h(g)| \leq |L_h(g-g_j)| + |L_h(g_j)|$. Hence to show that $L_h(g)$ tends to $0$ as $h$ tends to $0$, it follows from property $3$ that it is enough to show that $L_h(g-g_j) \to 0$ as $j \to \infty$. By property $1$ it suffices to prove that there is a constant $C$ which does not depend on $h$ such that for all $h$ in $E'$, $|L_h(g-g_j)| \leq C ||g-g_j||_p$. Moreover, since a bounded point derivation is already a bounded linear functional, it is enough to show that there is a constant $C$ which does not depend on $j$ such that $|\Delta_h^t (g(x_0)-g_j(x_0))| \leq C ||g-g_j||_p $. Furthermore, since the difference quotient is a finite linear combination of terms of the form $g(x_0 +sh) -g_j(x_0 +sh)$, it is enough to show that for each $s$ between $0$ and $t$, $|g(x_0 +sh) -g_j(x_0 +sh)| \leq C ||g-g_j||_p$. This is done in Lemma \ref{Browder3}; however, we will also need the following factorization lemma.

\begin{lemma}
\label{factorization}

Let $t$ be a positive integer. Then

\[\dfrac{1}{z-x} =\sum_{m = 1}^{t} \dfrac{(x-x_0)^{m-1}}{(z-x_0)^m}+ \dfrac{(x-x_0)^t}{(z-x)(z-x_0)^t}\]

\end{lemma}

\begin{proof}

The proof is by induction. For the base case, note that 

\begin{equation}
\label{factoreq}
\dfrac{1}{z-x} = \dfrac{1}{z-x_0}+ \dfrac{x-x_0}{(z-x)(z-x_0)}
\end{equation}
\bigskip

 \noindent Now assume that we have shown that

 \[ \dfrac{1}{z-x} =\sum_{m = 1}^{t-1} \dfrac{(x-x_0)^{m-1}}{(z-x_0)^m}+ \dfrac{(x-x_0)^{t-1}}{(z-x)(z-x_0)^{t-1}}\]

\bigskip

\noindent Then

 \[ \dfrac{1}{z-x}= \sum_{m = 1}^{t-1} \dfrac{(x-x_0)^{m-1}}{(z-x_0)^m}+ \dfrac{1}{z-x} \cdot\dfrac{(x-x_0)^{t-1}}{(z-x_0)^{t-1}}\]

\bigskip

\noindent and applying \eqref{factoreq} to the $\dfrac{1}{z-x}$ term in the sum proves the lemma. 

\end{proof}

\begin{lemma}
\label{Browder3}

 Suppose that $h$ belongs to $E'$ and let $j$ be a positive integer. Let $0 \leq s \leq t$. Then there exists a constant $C$ which does not depend on $h$ or $j$ such that $\dfrac{|g(x_0+sh)-g_j(x_0+sh)|}{|h|^t}  \leq C  ||g-g_J||_p$.

\end{lemma}

\begin{proof}

\bigskip

 Let $x = x_0 +sh$. Then $x$ belongs to $E$ and the hypotheses of Lemma \ref{Bishop1} are satisfied, so $k_x dA$ is a representing measure for $x$. Thus




\[ |g(x)-g_j(x)| = \dfrac{1}{|c|}\left| \int [g(z)-g_j(z)] \left(\dfrac{z-x_0}{z-x} \right)k(z) dA_z\right| \]

\bigskip

\noindent Since $D^0_{x_0} [g(z)-g_j(z)] = 0$, it follows that $\int [g(z)-g_j(z)] k(z) dA_z = 0$. Then since $\dfrac{z-x_0}{z-x} = 1 + \dfrac{x-x_0}{z-x} $, we obtain that

\begin{dmath}
\label{eq2.5}
|g(x)-g_j(x)| 
=  \dfrac{|x-x_0|}{|c|}\left| \int [g(z)-g_j(z)] \dfrac{k(z)}{z-x} dA_z\right|
\end{dmath}

\bigskip

 \noindent Applying Lemma \ref{factorization} to the $\dfrac{k(z)}{z-x}$ term in the rightmost integral in \eqref{eq2.5} shows that

 \begin{dmath}
 \label{eq3}
 |g(x)-g_j(x)|=  \dfrac{|x-x_0|}{|c|} \left|\sum_{m = 1}^{t} \int [g(z)-g_j(z)]\dfrac{(x-x_0)^{m-1}k(z)}{(z-x_0)^m }  dA_z \\+ \int [g(z)-g_j(z)] \dfrac{(x-x_0)^{t} k(z)}{(z-x)(z-x_0)^{t}}  dA_z \right|
 \end{dmath}

\bigskip

We can factor out the powers of $x-x_0$ from each integral since integration is with respect to $z$. Thus each integral in the sum is of the form $ \int [g(z)-g_j(z)] \dfrac{k(z)}{(z-x_0)^m}dA$ where $1 \leq m \leq t$. This integral simplifies to $ \int [g(z)-g_j(z)] \dfrac{(z-x_0)^{t-m} k_t(z)}{t!} dA_z$ and by Lemma \ref{Wilken2}, the integral reduces to a constant times the $m$-th order bounded point derivation of $g(z)-g_j(z)$, which is $0$ for $1 \leq m \leq t$. Hence $|g(x)-g_j(x)| = \displaystyle \dfrac{|x-x_0|}{|c|t!} \left|\int [g(z)-g_j(z)] \dfrac{(x-x_0)^{t}k_t(z)}{(z-x)}  dA_z \right|$ which simplifies to  $\displaystyle \dfrac{|x-x_0|^t}{|c|t!} \left|\int [g(z)-g_j(z)] \dfrac{(x-x_0)k_t(z)}{(z-x)}  dA_z \right| $. Finally by H\"{o}lder's inequality,

 \[ \dfrac{|x-x_0|^t}{|c|t!} \left|\int [g(z)-g_j(z)] \dfrac{(x-x_0)k_t(z)}{(z-x)}  dA_z \right| \leq \dfrac{|x-x_0|^t}{|c|t!} \left \|g-g_j \right \|_p \left \|\dfrac{(x-x_0)k_t}{(z-x)}\right \|_q\]

\bigskip

\noindent and since it follows from property $1$ of $E$ that $\left \|\dfrac{(x-x_0)k}{(z-x)(z-x_0)^{t}}\right \|_q \leq \delta_0$, there is a constant $C$ that does not depend on $h$ or $j$ such that

\[|g(x) - g_j(x)| \leq C |x-x_0|^t ||g-g_j||_p.\]

\bigskip

 \noindent Since $x = x_0 +sh$, it follows that $ |g(x_0+sh) - g_j(x_0+sh)| \leq C |s|^t |h|^t ||g-g_j||_p$ and thus 

\[\dfrac{|g(x_0+sh)-g_j(x_0+sh)|}{|h|^t}  \leq C  ||g-g_j||_p\]

\end{proof}

\bigskip

\section*{Acknowledgements}

I am grateful to Professor James Brennan for introducing me to these bounded point derivation problems and for the valuable assistance that he gave me in the preparation of this paper. I would also like to thank the referee for providing helpful comments and suggestions for improving this paper.


\end{document}